\documentclass[11pt,letterpaper,reqno]{amsart}

\usepackage{amsmath}
\usepackage{amsfonts}
\usepackage{amssymb}
\usepackage{amsthm}
\usepackage{amscd}

\usepackage[colorlinks=true]{hyperref}
\hypersetup{urlcolor=blue, citecolor=blue}

\usepackage{latexsym}
\usepackage{mathrsfs}
\usepackage{psfrag}
\usepackage[dvips]{epsfig}
\usepackage{epsfig}
\usepackage[all]{xy}         
\usepackage{wrapfig}

\usepackage[margin=2.5cm]{geometry}

\swapnumbers 

\theoremstyle{plain}                              
\newtheorem{thm}{Theorem}[section]
\newtheorem{prop}[thm]{Proposition}

\newtheorem{lem}[thm]{Lemma}
\newtheorem{cor}[thm]{Corollary}

\newtheorem*{mainthm}{Theorem}               

\theoremstyle{definition}                         

\newtheorem*{remark}{Remarks} 

\theoremstyle{remark}                             

\numberwithin{equation}{section}

\newcommand{\R}{\mathbb{R}}                     
\newcommand{\Z}{\mathbb{Z}}                     

\newcommand{\ms}[1]{\mathscr{#1}}               

\newcommand{\veps}{\varepsilon}        

\newcommand{\del}{\partial}

\providecommand{\norm}[1]{\left\lVert#1\right\rVert}       

\providecommand{\abs}[1]{\left\lvert#1\right\rvert}        

\begin{document}

\title[Cross sections]{Global cross sections for Anosov flows}


\author[S. N. Simi\'c]{Slobodan N. Simi\'c}

\address{Department of Mathematics and Statistics, San Jos\'e State University, San
  Jos\'e, CA 95192-0103}



\email{simic@math.sjsu.edu}


\date{\today}
\dedicatory{}



\begin{abstract}
  We provide a new criterion for the existence of a global cross
  section to a volume-preserving Anosov flow. The criterion is
  expressed in terms of expansion and contraction rates of the flow
  and is more general than the previous results of similar kind.
\end{abstract}

\maketitle





\section{Introduction}

Henri Poincar\'e introduced the idea of a cross section to a flow to
study the 3-body problem. A global cross section to a flow $\Phi$ on a
manifold $M$ is a codimension one submanifold $\Sigma$ of $M$ such
that $\Sigma$ intersects every orbit of $\Phi$ transversely. It is
natural to ask whether any given non-singular flow admits one.

If $\Sigma$ is a global cross section for $\Phi$, it is not hard to
check that every orbit which starts on $\Sigma$ returns to $\Sigma$
after some positive time, defining the Poincar\'e first-return map $g
: \Sigma \to \Sigma$. The analysis of $\Phi$ can then be reduced to
the study of the map $g$, which in principle can be an easier
task. The flow can be reconstructed from the Poincar\'e map by
suspending it (cf., ~\cite{katok+95}).

The object of this paper is to investigate the existence of global
cross sections to volume-preserving Anosov flows.

Recall that a non-singular flow $\Phi = \{ f_t \}$ on a closed
(compact and without boundary) Riemannian manifold $M$ is called
\textsf{Anosov} if there exists an invariant splitting $TM = E^{ss} \oplus E^c
\oplus E^{uu}$ of the tangent bundle of $M$ and uniform constants $c >
0$, $0 < \mu_- \leq \mu_+ < 1$ and $\lambda_+ \geq \lambda_- > 1$ such
that the \textsf{center bundle} $E^c$ is spanned by the infinitesimal
generator $X$ of the flow and for all $v \in E^{ss}$, $w \in E^{uu}$,
and $t \geq 0$, we have
\begin{equation}      \label{eq:ss}
  \frac{1}{c} \mu_-^t \norm{v} \leq \norm{Tf_t (v)} \leq c \mu_+^t \norm{v},
\end{equation}
and
\begin{equation}   \label{eq:uu}
  \frac{1}{c} \lambda_-^t \norm{w} \leq \norm{Tf_t (v)} \leq c \lambda_+^t \norm{w},
\end{equation}
where $Tf_t$ denotes the derivative (or tangent map) of $f_t$. We call
$E^{ss}$ and $E^{uu}$ the \textsf{strong stable} and \textsf{strong
  unstable bundles}; $E^{cs} = E^c \oplus E^{ss}$ and $E^{cu} = E^c
\oplus E^{uu}$ are called the \textsf{center stable} and
\textsf{center unstable bundles}. It is well-known
\cite{hps77,hassel+94} that all of them are H\"older continuous and
uniquely integrable \cite{anosov+67}. The corresponding foliations
will be denoted by $W^{ss}, W^{uu}, W^{cs}$, and $W^{cu}$. They are
also H\"older continuous in the sense that each one admits H\"older
foliation charts. This means that if $W^\sigma$ ($\sigma \in \{ss, uu,
cs, cu \}$) is $C^\theta$, then every point in $M$ lies in a
$C^\theta$ chart $(U,\varphi)$ such that in $U$ the local
$W^\sigma$-leaves are given by $\varphi_{k+1} = \text{constant},
\ldots, \varphi_n = \text{constant}$, where $\varphi =
(\varphi_1,\ldots,\varphi_n)$ is a $C^\theta$ homeomorphism and $k$ is
the dimension of $W^\sigma$. The leaves of all invariant foliations
are as smooth as the flow. See also \cite{psw+97} for a discussion of
regularity of H\"older foliations.


\noindent
\paragraph{\textbf{Related work.}}
\label{sec:textbfprevious-work}
The first results on the existence of global cross sections to Anosov
flows were proved by Plante in \cite{plante72}. He showed that if
$E^{ss} \oplus E^{uu}$ is a uniquely integrable distribution or
equivalently, if the foliations $W^{ss}$ and $W^{uu}$ are jointly
integrable\footnote{This means that locally speaking the
  $W^{uu}$-holonomy between local $W^{cs}$-leaves takes local
  $W^{ss}$-leaves to $W^{ss}$-leaves.}, then the Anosov flow admits a
global cross section. Sharp \cite{sharp+93} showed that a transitive
Anosov flow admits a global cross section if it is not homologically
full; this means that for every homology class $\alpha \in H_1(M,\Z)$
there is a closed $\Phi$-orbit $\gamma$ whose homology class equals
$\alpha$. (This is equivalent to the condition that there is \emph{no}
fully supported $\Phi$-invariant ergodic probability measure whose
asymptotic cycle in the sense of Schwartzman~\cite{schwartz+57} is
trivial.) Along different lines Bonatti and Guelman~\cite{bg+09}
showed that if the time-one map of an Anosov flow can be $C^1$
approximated by Axiom A diffeomorphisms, then the flow is
topologically equivalent to the suspension of an Anosov
diffeomorphism.

Let $n_s = \dim E^{ss}$ and $n_u = \dim E^{uu}$. If $n_u = 1$ or $n_s
= 1$, the Anosov flow is said to be of codimension one. In the
discussion that follows we always assume $n_u = 1$. In \cite{ghys89}
Ghys proved the existence of global cross sections for codimension one
Anosov flows in the following cases: (1) if $E^{su} = E^{ss} \oplus
E^{uu}$ is $C^1$ and $n \geq 4 $ (in this case the global cross
section has constant return time); if (2) the flow is
volume-preserving, $n \geq 4$ and $W^{cs}$ is of class $C^2$. This was
generalized by the author in \cite{sns+96} and \cite{sns+97} where we
showed that a codimension one Anosov flow admits a global cross
section if any of the following assumptions is satisfied: (1) $E^{su}$
is Lipschitz (in the sense that it is locally spanned by Lipschitz
vector fields) and $n \geq 4$; (2) the flow is volume-preserving, $n
\geq 4$, and $E^{su}$ is $C^\theta$-H\"older for \emph{all} $\theta <
1$ (3) the flow is volume-preserving, $n \geq 4$, and $E^{cs}$ is of
class $C^{1 + \theta}$ for \emph{all} $\theta < 1$. Note that all the
regularity assumptions above require that the invariant bundles be
smoother than they usually are: $E^{su}$ is generically only H\"older
continuous and in the codimension one case, $E^{cs}$ is generically
only $C^{1 + \theta}$ for some small $0 < \theta < 1$. See
\cite{hassel+94, hassel+97} and \cite{hasselblatt-wilkinson+99}.

The goal of this paper is to establish the following result.

\begin{mainthm}
  Let $\Phi = \{ f_t \}$ be a volume-preserving Anosov flow on a
  closed Riemannian manifold $M$ and let $0 < \theta \leq 1$ be the smaller
  of the H\"older exponents of $W^{ss}$ and $W^{cs}$. If
  \begin{equation}     \label{eq:mainthm}
    \mu_+^{(n_s-1)\theta} \lambda_+^{(n_u-1)\theta} < \mu_-^{2(1-\theta)},
  \end{equation}
  then $\Phi$ admits a global cross section.
\end{mainthm}

\begin{remark}
  \begin{itemize}
  \item[(a)] The condition \eqref{eq:mainthm} has a chance of being
    satisfied only if $n_u$ is much smaller than $n_s$. If $n_u >
    n_s$, then by reversing time it is easy to show that
    \begin{displaymath}
      \lambda_+^{2(1-\alpha)} < \lambda_-^{(n_u-1) \alpha} \mu_-^{(n_s - 1) \alpha}
    \end{displaymath}
    also implies the existence of a global cross section, where
    $\alpha$ is the minimum of the H\"older exponents of $E^{uu}$ and $E^{cu}$.

  \item[(b)] If the flow is of codimension one with $n_u = 1$, then
    \eqref{eq:mainthm} reduces to
    \begin{equation}    \label{ref:codim}
      \mu_+^{(n-3)\theta} < \mu_-^{2(1-\theta)}.
    \end{equation}
    It is well-known (cf., \cite{hassel+94} and \cite{hps77}) that the center stable bundle $E^{cs}$ and strong
    unstable bundle $E^{uu}$ of a volume-preserving Anosov flow in
    dimensions $n \geq 4$ are both $C^{1 + \text{H\"older}}$. Thus if
    $E^{su}$ is Lipschitz as in \cite{sns+96} or $C^\theta$, for all
    $\theta < 1$, as in \cite{sns+97}, then \eqref{ref:codim} is
    clearly satisfied. If $E^{cs}$ is $C^{1+\theta}$ for all $\theta <
    1$ as in \cite{sns+97}, then it is not hard to show that $E^{ss}$
    is necessarily of class $C^\theta$ for all $\theta < 1$, which
    again implies \eqref{ref:codim}. Therefore, in the case of
    volume-preserving codimension one Anosov flows, our result implies all the
    previously known criteria for the existence of global cross
    sections.

  \item[(c)] In the early 1970's, prompted by a dearth of examples,
    A. Verjovsky conjectured~\cite{verj74} that every codimension one
    Anosov flow in dimensions $n \geq 4$ admits a global cross
    section. The importance of Verjovsky's conjecture stems from the
    fact that codimension one Anosov diffeomorphisms were classified
    by Franks~\cite{franks70} and Newhouse~\cite{newhouse70} who
    showed that every such diffeomorphism is topologically conjugate
    to a linear hyperbolic automorphism of a torus. Therefore, the
    affirmation of Verjovsky's conjecture would yield a complete
    classification of codimension one Anosov flows in dimensions $n
    \geq 4$.

    Progress towards Verjovsky's conjecture was made in the early
    1980's by Plante~\cite{plante81,plante83} and
    Armendariz~\cite{armendariz} who showed that the conjecture holds
    if the fundamental group of the manifold is solvable. By the work
    of Asaoka~\cite{asaoka+08} it follows that it suffices to prove
    the conjecture for volume-preserving flows, since any
    topologically transitive codimension one Anosov flow is
    topologically equivalent to a volume-preserving one. As of this
    writing, the conjecture remains open.

   \end{itemize}
\end{remark}

Throughout this paper smooth will mean of class $C^\infty$. \\

\noindent
\paragraph{\textbf{Outline of the proof.}}
The main idea of the proof of the theorem is to find a smooth closed
1-form $\eta$ such that $\eta(X) > 0$, where $X$ is the infinitesimal
generator of the Anosov flow. It is not hard to see that this
immediately implies the existence of a global cross section (cf.,
\S\ref{sec:proof}). To construct $\eta$, we use the fact that for any
$k$-form $\xi$, the $C^0$-distance from $\xi$ to the space of closed
$k$-forms is bounded above by the $C^0$ norm $\norm{d\xi}$; see
Proposition~\ref{prop:forms}. It therefore suffices to construct a
smooth 1-form $\xi$ such that $\abs{\xi(X_p)} > \norm{d\xi}$, for all
$p \in M$, since Proposition~\ref{prop:forms} then yields a smooth
closed 1-form $\eta$ such that $\eta(X) > 0$. We will actually
construct a smooth 1-form $\xi$ such that $\xi(X) = 1$ and
$\norm{d\xi} < 1$.

The construction of $\xi$ is divided into two steps. In the first
step, we find an initial candidate for $\xi$ such that the norm of its
exterior derivative restricted to $E^{cs}$ is small, while its total
norm blows up in a way controlled by the H\"older exponent
$\theta$. More precisely, for each $\veps > 0$ we construct a smooth
1-form $\xi_0^\veps$ on $M$ such that $\xi_0^\veps(X) = 1$,
$\norm{d\xi_0^\veps \!  \restriction_{E^{cs}}} \leq D \veps^\theta$,
and $\norm{d \xi_0^\veps} \leq K \veps^{\theta-1}$, where $D$ and $K$
are positive constants independent of $\veps$. This is achieved by
carefully building smooth local cross sections and the corresponding
flow boxes; cf., \S\ref{sec:local-sections}.

In the second step, we pull back $\xi_0^\veps$ by $f_{-t}$ for
suitable $t > 0$ to make the norm of its exterior derivative small in
the remaining directions. For this, we use an estimate (see
Lemma~\ref{lem:backward}) on the growth of $\norm{Tf_{-t}(v \wedge
  w)}$, for $t > 0$, $v \in E^{ss}$ and $w \in E^{uu}$. Assuming
\eqref{eq:mainthm}, we then show that there exist $\veps > 0$ and $t >
0$ such that $\xi = f_{-t}^\ast \xi_0^\veps$ has the desired properties.

\section{Preliminaries}
\label{sec:preliminaries}

\subsection{Regularization}
\label{sec:reg}

Here we recall a standard technique for approximating locally
integrable functions by smooth ones called
\textsf{regularization} or \textsf{mollification} (see
\cite{evans+98,stein+70}). Define the standard mollifier $\eta : \R^n
\to \R$ by
\begin{displaymath}
      \eta(x) = \begin{cases}
      A_0 \exp \left( \frac{1}{\abs{x}^2 - 1} \right) & 
      \text{if $\abs{x} < 1$} \\
      0 & \text{if $\abs{x} \geq 1$},
      \end{cases}
\end{displaymath}
where $A_0$ is chosen so that $\int \eta \: dx = 1$. For every $\veps
> 0$, set $\eta_\veps(x) = \veps^{-n} \eta(x/\veps)$. Note that the
support of $\eta_\veps$ is contained in the ball $B(0,\veps)$ of
radius $\veps$ centered at the origin and that $\int \eta_\veps \: dx
= 1$. For a locally integrable function $u : \R^n \to \R$ define
\begin{displaymath}
  u^\veps(x) = (u \ast \eta_\veps)(x) = \int_{\R^n} u(y)
  \eta_\veps(x-y) \: dy = \int_{\R^n} u(x-y) \eta_\veps(y) \: dy.
\end{displaymath}

\begin{prop}  \label{prop:reg}
  
  Assume that $u: \R^n \to \R$ is locally integrable. Then:
\begin{enumerate}

  \item[(a)] $u^\veps \in C^\infty(\R^n)$.
  \item[(b)] If $u \in L^\infty$, then
    $\norm{u^\veps}_{L^\infty} \leq \norm{u}_{L^\infty}$.
  \item[(c)] If $u$ is continuous, then $u^\veps \to u$, uniformly as
    $\veps \to 0$. If $u \in C^\theta$ $(0 < \theta \leq 1)$, then
    $\norm{u^\veps - u}_{C^0} \leq \norm{u}_{C^\theta} \:
    \veps^\theta$, where the $C^\theta$-norm is the sum of its
    sup-norm and its best H\"older constant.
  \item[(d)] If $u \in C^\theta$, then 
    \begin{displaymath}
      \norm{du^\veps} \leq
      \norm{d\eta}_{L^1} \norm{u}_{C^\theta} \veps^{\theta-1}
    \end{displaymath}
    where $\norm{d\eta}_{L^1} = \max_i \int_{\R^n} \abs{\del \eta/\del
      x_i} dx$ and $\norm{du^\veps}$ denotes the maximum of the
    sup-norms of the partial derivatives of $u^\veps$.
\end{enumerate}

\end{prop}
\begin{proof}
  
  Proofs of (a), (b) and the first part of (c) can be found in
  \cite{evans+98}. For the second part of (c), we
  have
  \begin{align*}
    \abs{u^\veps(x) - u(x)} & = \left\lvert \int_{B(0,\veps)} 
      \eta_\veps(y) [u(x-y) - u(x)] \, dy \right\rvert \\
    & \leq \norm{u} \veps^\theta \int_{B(0,\veps)}
    \eta_\veps(y) \, dy \\
      & = \norm{u}_{C^\theta} \veps^\theta.
  \end{align*}
  If $u \in C^1$, then the same estimates hold with $\theta$ replaced
  by $1$. 
  
  Observe that since $\eta_\veps$ has compact support,
  \begin{equation}                    \label{eq:compact}
    \int_{\R^n} \frac{\del \eta_\veps}{\del x_i}(y) \, dy = 0,
  \end{equation}
  for $1 \leq i \leq n$. Note also that

  \begin{equation*}
    \frac{\del \eta_\veps}{\del x_i}(x) = \frac{1}{\veps^{n+1}} 
    \frac{\del \eta}{\del x_i}\left( \frac{x}{\veps} \right).
  \end{equation*}
  Assuming $u \in C^\theta$, we obtain (d):
  \begin{align*}
    \left\lvert \frac{\del u^\veps}{\del x_i}(x) \right\rvert & =
    \left\lvert \int_{\R^n} u(x-y) \frac{\del \eta_\veps}{\del x_i}(y)
      \, dy \right\rvert \\
    & \overset{\text{by} \ \eqref{eq:compact}}{=} \left\lvert
      \int_{B(0,\veps)} [u(x-y) - u(x)] \frac{\del \eta_\veps}{\del
        x_i}(y) \, dy \right\rvert,  \\
    & \leq \norm{u}_{C^\theta} \, \veps^\theta \int_{B(0,\veps)}
      \left\lvert \frac{\del \eta_\veps}{\del x_i}(y) \right\rvert \, dy  \\
      & = \norm{u}_{C^\theta} \, \veps^\theta \int_{B(0,\veps)}
      \frac{1}{\veps^{n+1}} \left\lvert \frac{\del \eta}{\del
          x_i}\left(\frac{y}{\veps}\right)
      \right\rvert \, dy  \\
      & \overset{z = \frac{y}{\veps}}{=} \norm{u}_{C^\theta} \,
      \veps^\theta \cdot \frac{1}{\veps} \int_{B(0,1)} \left\lvert
        \frac{\del \eta}{\del x_i}(z) \right\rvert \, dz,  \\
    & \leq \norm{d\eta}_{L^1} \norm{u}_{C^\theta} \, \veps^{\theta-1}. \qedhere
  \end{align*}
\end{proof}

\begin{cor}   \label{cor:mollify}
  Let $M$ be a compact manifold without boundary. Fix a finite atlas
  $\ms{A} = \{ (U_i,\varphi_i) : i \in I \}$ of $M$. If $u : U \to \R$
  is $C^\theta$ and $U \subset U_j$ for some $j \in I$, then there
  exist $\veps_\ast > 0$ and a family of smooth approximations
  $u^\veps$ $(0 < \veps < \veps_\ast)$ of $u$ such that:

  \begin{enumerate}

  \item[(a)] $u^\veps$ is defined on $U^\veps \subset U$ where
    $U^\veps = \varphi_j^{-1}\{ x \in \varphi_j(U) : d(x, \del
    \varphi_j(U)) > \veps \}$.

  \item[(b)] $\norm{u^\veps - u}_{C^0} \leq \kappa \veps^\theta$.

  \item[(c)] $\norm{du^\veps}_{C^0} \leq \kappa \veps^{\theta - 1}$.

  \end{enumerate}

  Here $\kappa > 0$ depends only on $\ms{A}$ and
  $\norm{u}_{C^\theta}$. 
\end{cor}

\begin{proof}
  The family $u^\veps = [(u \circ \varphi_j^{-1}) \ast \eta_\veps]
  \circ \varphi_j$ has the desired properties. We can take
  $\veps_\ast$ to be any positive number such that for $\veps <
  \veps_\ast$, the sets $U^\veps$ defined above are non-empty.
\end{proof}

\subsection{Construction of local cross sections}
\label{sec:local-sections}

In this section we construct a finite covering of the manifold by smooth flow boxes defined by carefully chosen smooth local cross sections. 

Let $p \in M$ be arbitrary and choose a $W^{ss}$-foliation chart $(U,\varphi)$ containing $p$. This means that $\varphi : U \to \R^n$ is a $C^\theta$-homeomorphism such that the local $W^{ss}$-leaves in $U$ are given by
\begin{displaymath}
  \varphi_{n_s+1} = \text{constant}, \cdots, \varphi_n = \text{constant},
\end{displaymath}
where $\varphi = (\varphi_1, \ldots, \varphi_n)$. We can also arrange that the local $W^{cs}$-leaves in $U$ are defined by
\begin{displaymath}
  \varphi_{n_s+1} = \text{constant}, \cdots, \varphi_{n-1} = \text{constant},
\end{displaymath}
so that in each local $W^{cs}$-leaf in $U$ $W^{ss}$ is given by $\varphi_n = \text{constant}$. Furthermore, we can take $U$ so that its closure is contained in a \emph{smooth} chart for $M$. This condition will allow us to mollify continuous functions defined on $U$ without having to shrink the domain.

Even though $\varphi$ is only H\"older, the flow invariance of $W^{ss}$ implies that each $\varphi_i$ is differentiable with respect to $X$. Since $X$ is tangent to the $W^{cs}$ leaves, it follows that $X \varphi_i = 0$, for $n_s + 1 \leq i \leq n-1$. Furthermore, the restriction of $W^{ss}$ to $W^{cs}$-leaves is as smooth as the flow, so $\varphi_n$ is smooth on the local $W^{cs}$-leaves in $U$. Since $X$ is uniformly transverse to $W^{ss}$ it is clear that $X\varphi_n \neq 0$ and by continuity there exists $\delta > 0$ such that $\abs{X \varphi_n} \geq \delta$ on $U$.

Define 
\begin{displaymath}
  \Sigma_0 = \bigcup_{x \in W^{uu}_{\text{loc}}(p)} W^{ss}_{\text{loc}}(x),
\end{displaymath}
where $W^\sigma_{\text{loc}}(x)$ denotes the local $W^\sigma$-leaf in $U$ (for $\sigma \in \{ uu, ss \}$). Then $\Sigma_0$ is a H\"older continuous local cross section for the flow. (This makes sense, since the intersection of $\Sigma_0$ with each local $W^{cs}$-leaf is a local $W^{ss}$-leaf, hence smooth and transverse to the flow.) Let $\Sigma$ be a slightly smaller compact subset of $\Sigma_0$ containing $p$.

We claim that there exists $T > 0$, depending on $\Sigma$ and $U$, such that 
\begin{displaymath}
  V \stackrel{\text{def}}{=} \bigcup_{\abs{t} < T} f_t(\Sigma) 
\end{displaymath}
is a H\"older continuous flow box contained in $U$. Assume the contrary; it follows that the map $(t,q) \mapsto f_t(q)$ fails to be 1--1 on $(-T,T) \times \Sigma$, for any $T > 0$. Thus there exist sequences $(q_k)$ and $(T_k)$ such that $q_k \in \Sigma$, $T_k > 0$, $T_k \to 0$ and $f_{T_k}(q_k) \in \Sigma$, for all $k$. Since $f_{T_k}(q_k)$ and $q_k$ lie in the same local $W^{cs}$-leaf, it follows that $\varphi_n(f_{T_k}(q_k)) = \varphi_n(q_k)$. On the other hand, by compactness of $\Sigma$, $(q_k)$ has a subsequence $(q_{k_j})$ which converges to some $q \in \Sigma_0$. This implies $X\varphi_n(q) = 0$, which contradicts $\abs{X\varphi_n} \geq \delta$. Therefore, $T$ exists; we will call it the \emph{length} of the continuous flow box $V$ (although a more appropriate name would be half-length).

Define $\tau : V \to \R$ by
\begin{displaymath}
  \tau(f_t q) = t,
\end{displaymath}
for $q \in \Sigma$ and $\abs{t} < T$. It is clear that $\tau$ is $C^\theta$, $X\tau = 1$ and $\tau$ is constant on the local $W^{ss}$-leaves in $V$.

Next we approximate $\Sigma, \tau$ and $V$ by smooth objects with similar properties.

\begin{lem}    \label{lem:tau}
  There exists $\veps_\ast > 0$ such that for every $0 < \veps < \veps_\ast$ there exist an open set
  $V^\veps \subset V$ and a smooth function $\tau^\veps : V^\veps \to
  \R$ with the following properties:

  \begin{enumerate}

  \item[(a)] $X\tau^\veps = 1$.

  \item[(b)] $\abs{d_q \tau^\veps(v)} \leq A \veps^\theta \norm{v}$, for all $q \in V^\veps$
    and $v \in E^{ss}$, where $A$ is a constant independent of $\veps$, $q$, and $v$.

  \item[(c)] $\norm{d\tau^\veps} \leq B \veps^{\theta-1}$, where $B$
    is a constant independent of $\veps$.

  \item[(d)] The Hausdorff distance between $V$ and $V^\veps$ tends to zero, as $\veps \to 0$.

  \end{enumerate}

\end{lem}

  \begin{proof}
    For the sake of notational simplicity, we will write $f(x) \lesssim g(x)$ ($x \in S$) to mean that there exists a constant $C$ independent of $x \in S$ such that $f(x) \leq C g(x)$, for all $x \in S$.

    Since $U$ is contained in a smooth chart for $M$, by Corollary~\ref{cor:mollify} there exists $\veps_\ast$ such that for each $i > n_s$ there is a family $\varphi_i^\veps$ ($0 < \veps < \veps_\ast$) of smooth approximations of $\varphi_i$ satisfying
\begin{equation}  \label{eq:phi}
  \norm{\varphi_i^\veps - \varphi_i}_{C^0} \lesssim \veps^\theta \qquad \text{and}
  \qquad \norm{d\varphi_i^\veps}_{C^0} \lesssim \veps^{\theta-1}.
\end{equation}
Note that $u^\veps$ is defined on all of $U$. Denote by $\ms{F}_\veps$ the foliation of $U^\veps$ defined by
\begin{displaymath}
  \varphi^\veps_{n_s+1} = \text{constant}, \cdots, \varphi^\veps_n = \text{constant}.
\end{displaymath}
It is easy to see that $\ms{F}_\veps$ is smooth and the (largest principal) angle between $W^{ss}_\text{loc}$ and $\ms{F}_\veps$ is $\lesssim \veps^\theta$. Let
\begin{displaymath}
  \Sigma^\veps = \bigcup_{x \in W^{uu}_{\text{loc}}(p)} \ms{F}_\veps(x).
\end{displaymath}
Then $\Sigma^\veps$ is a smooth local cross section and there exists $T_\veps > 0$ such that the set
\begin{displaymath}
  V^\veps = \bigcup_{\abs{t} < T_\veps} f_t(\Sigma^\veps)
\end{displaymath}
is a smooth flow box for $X$ contained in $U^\veps$. It is clear that the length $T_\veps$ of $V^\veps$ is close to the length $T$ of $V$; we can also take $T_\veps \leq T$.

Define $\tau^\veps : V^\veps \to \R$ by
\begin{displaymath}
  \tau^\veps(f_tq) = t,
\end{displaymath}
for all $q \in \Sigma^\veps$ and $-T_\veps < t < T_\veps$. Clearly, $\tau^\veps$ is smooth and $X\tau^\veps = 1$, proving (a).

\begin{figure}[h]
\centerline{
\includegraphics[width=0.7\hsize]{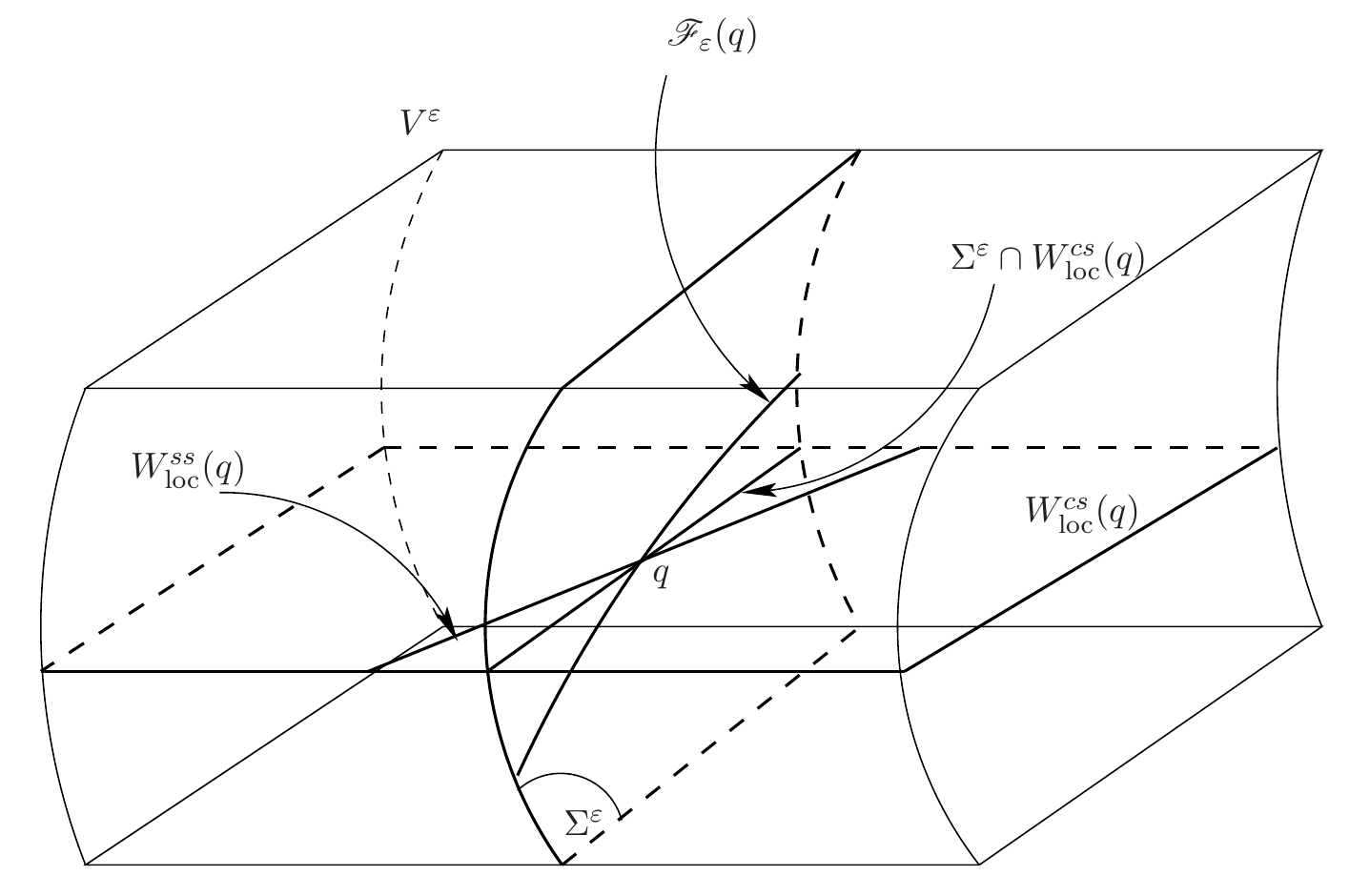}}
\caption{The flow box $V^\veps$ with local cross section $\Sigma^\veps$.}
\label{fig:flowbox}
\end{figure}

We will first show that (b) holds along $\Sigma^\veps$.  Let $q \in \Sigma^\veps$ be arbitrary. Since the angle between $\ms{F}_\veps(q)$ and $W^{ss}_{\text{loc}}(q)$ is $\lesssim \veps^\theta$, it follows that the angle between $\Sigma^\veps \cap W^{cs}_{\text{loc}}(q)$ and $W^{ss}_{\text{loc}}(q)$ is also $\lesssim \veps^\theta$. See Figure~\ref{fig:flowbox}. Then the fact that $d\tau^\veps = 0$ on $\Sigma^\veps \cap W^{cs}_{\text{loc}}(q)$ implies $\norm{d_q\tau^\veps \! \restriction_{E^{ss}}} \lesssim \veps^\theta$, as desired. To extend this to all points in $V^\veps$ we will use the invariance of $d\tau^\veps$ with respect to the flow: $f_t^\ast (d\tau^\veps) = d\tau^\veps$, whenever both sides are defined. For $q \in \Sigma^\veps$, $\abs{t} < T_\veps$ and $v \in E^{ss}_{f_t q}$, we have
\begin{align*}
  \abs{d_{f_t q}\tau^\veps(v)} & = \abs{d_q \tau^\veps(Tf_{-t}(v))}  \\
    & \leq \norm{d_q\tau^\veps \! \restriction_{E^{ss}}} \norm{Tf_{-t}(v)} \\
    & \lesssim \veps^\theta \mu_-^{-t} \norm{v} \\
    & \leq \veps^\theta \mu_-^{-T_\veps} \norm{v} \\
    & \leq \veps^\theta \mu_-^{-T} \norm{v},
\end{align*}
where we used $T_\veps \leq T$. This proves (b) for all $q \in V^\veps$.

To prove (c), we work in a smooth coordinate system in which $X = \del/\del x_1$. Since $\Sigma$ is a $C^\theta$ hypersurface, it is locally the graph of a $C^\theta$ function $g$ such that $\tau(g(z),z) = 0$, for all $z$ is some open set in $\R^{n-1}$. Similarly, $\Sigma^\veps$ is locally the graph of a smooth function $g^\veps$ such that $\tau^\veps(g^\veps(z),z) = 0$, for all $z$ in some open set in $\R^{n-1}$. Since $g^\veps \to g$, as $\veps \to 0$, \eqref{eq:phi} forces $\norm{dg^\veps} \lesssim \veps^{\theta-1}$. Differentiating $\tau^\veps(g^\veps(z),z) = 0$ with respect to $z_i$ for $i > 1$ and using $\del \tau^\veps/\del z_1 = X\tau^\veps = 1$, we obtain
\begin{displaymath}
  \abs{ \frac{\del \tau^\veps}{\del z_i}} = \abs{ \frac{\del g^\veps}{\del z_i}} \lesssim \veps^{\theta-1}.
\end{displaymath}
Thus (c) holds on $\Sigma^\veps$; we can extend it to $V^\veps$ using the flow invariance of $d\tau^\veps$ as in the proof of (b).

Part (d) holds by construction. 
  \end{proof}

  By the above analysis and compactness, we can cover $M$ by finitely many H\"older flow boxes $V_1,\ldots, V_\ell$, each of which is equipped with a local H\"older cross section $\Sigma_i$. We can approximate each $\Sigma_i$ by a smooth cross section $\Sigma_i^\veps$ as above and obtain smooth flow boxes $V_i^\veps \subset V_i$ and smooth functions $\tau_i^\veps : V_i^\veps \to \R$ satisfying the properties from Lemma~\ref{lem:tau}; namely,
  \begin{equation}  \label{eq:tau}
    X\tau_i^\veps = 1, \qquad \norm{d\tau_i^\veps \! \restriction_{E^{ss}}} \leq A_i \veps^\theta, 
    \qquad \norm{d\tau_i^\veps} \leq B_i \veps^{\theta-1},
  \end{equation}
  where the constants $A_i, B_i$ are independent of $\veps$. Let $A = \max A_i, B = \max B_i$ (not to be confused with the constants $A, B$ in Lemma~\ref{lem:tau}). By Lemma~\ref{lem:tau} (d), there exists $\veps_\ast > 0$ such that for all $0 < \veps < \veps_\ast$ the sets $V_1^\veps,\ldots, V_\ell^\veps$ cover $M$.

\subsection{Distance to the space of closed forms}
\label{sec:closed-forms}

We will consider $C^r$ differential $k$-forms $\xi$, with $r \geq
1$. We denote the $C^0$ norm of $\xi$ on $M$ by $\norm{\xi}$:
\begin{equation}  \label{eq:def-norm}
  \norm{\xi} = \sup_{p \in M} \abs{\xi_p},
\end{equation}
where $\abs{\xi_p}$ is is the operator norm of $\xi_p$ as a $k$-linear
map $T_p M \times \cdots \times T_p M \to \R$.

Consider first a differential form $\omega$ on $M \times [0,1]$, where
$t$ is the coordinate in $[0,1]$. Denote by $\pi_M : M \times [0,1]
\to M$ and $\pi_I : M \times [0,1] \to [0,1]$ the obvious
projections. Since
\begin{displaymath}
  T_{(p,t)} (M \times [0,1]) = T_p M \oplus T_t [0,1],
\end{displaymath}
any differential $k$-form on $M \times [0,1]$ can be uniquely written
as
\begin{displaymath}
  \omega = \omega_0 + dt \wedge \eta,
\end{displaymath}
where $\omega_0(v_1,\ldots,v_k) = 0$ if some $v_i$ is in the kernel of
$(\pi_M)_\ast$ and $\eta$ is a $(k-1)$-form with the analogous
property (i.e., $i_v \omega_0 = i_v \eta = 0$, for every ``vertical''
vector $v \in T(M \times [0,1])$, where $i_v$ denotes contraction by
$v$).

Define a $(k-1)$-form $\ms{H}(\xi)$ on $M$ by
\begin{displaymath}
  \ms{H}(\omega)_p = \int_0^1 j_t^\ast \eta_{(p,t)} \: dt,
\end{displaymath}
where $j_t : M \to M \times [0,1]$ is defined by $j_t(p) =
(p,t)$. It is well-known (cf., \cite{spivak:cidg+05}) that
\begin{equation}   \label{eq:poincare}
  j_1^\ast \omega - j_0^\ast \omega = d(\ms{H}\omega) + \ms{H}(d\omega).
\end{equation}
Considering $\ms{H}$ as a linear operator from $\Omega^k(M \times
[0,1])$ to $\Omega^{k-1}(M)$, both equipped with the $C^0$ norm, it is
not hard to see that
\begin{equation}   \label{eq:norm}
  \norm{\ms{H}} = 1.
\end{equation}
We claim:

\begin{prop}  \label{prop:forms}
  Let $\xi$ be a $C^r$ differential $k$-form ($r \geq 1$) on a closed
  manifold $M$. Then
  \begin{displaymath}
    \inf \{ \norm{\xi - \eta} : \eta \in C^r, \ d\eta = 0\} \leq \norm{d\xi}.
  \end{displaymath}
\end{prop}

In other words, $\norm{d\xi}$ is an upper bound on the distance from
$\xi$ to the space of closed forms. The inequality also holds for
continuous forms which admit a continuous exterior derivative.

\begin{proof}[Proof of the Proposition]
  First, let us show that the result holds on any manifold $M$ which
  is smoothly contractible to a point $p_0$ via $H : M \times [0,1]
  \to M$, where $H(p,0) = p_0$ and $H(p,1) = p$, for all $p \in
  M$. Since $H \circ j_1$ is the identity map of $M$ and $H \circ j_0$
  is the constant map $p_0$, it follows that
  \begin{displaymath}
    \xi = (H \circ j_1)^\ast \xi = j_1^\ast (H^\ast \xi) \quad
    \text{and} \quad
    0 = (H \circ j_0)^\ast \xi = j_0^\ast (H^\ast \xi).
  \end{displaymath}
  Applying \eqref{eq:poincare} to $H^\ast \xi$, we obtain
  \begin{align*}
    \xi & = \xi - 0 \\
      & = j_1^\ast (H^\ast \xi) - j_0^\ast (H^\ast \xi) \\
      & = d \ms{H}(\xi) + \ms{H}(d \xi).
  \end{align*}
  Using \eqref{eq:norm}, we obtain
  \begin{displaymath}
    \norm{\xi - d \ms{H}(\xi)} = \norm{\ms{H}(d\xi)} \leq \norm{d\xi}.
  \end{displaymath}
  Therefore, the statement of the theorem holds for contractible $M$.

  Let $M$ now be any closed manifold and $\xi$ a $C^r$ $k$-form on
  $M$, $r \geq 1$. Cover $M$ by contractible open sets $U_1, \ldots,
  U_m$. Denote the operator $\ms{H}$ restricted to forms on $U_i$ by
  $\ms{H}_i$ and let $\xi_i$ be the restriction of $\xi$ to
  $U_i$. Define a $k$-form $\eta$ on $M$ by requiring that the
  restriction of $\eta$ to $U_i$ be equal to $d \ms{H}_i (\xi_i)$. We
  claim that $\eta$ is well-defined and closed.

  Indeed, $\xi_i = \xi_j$ on $U_i \cap U_j$, and $\ms{H}_i (\omega) =
  \ms{H}_j(\omega)$ for every $k$-form $\omega$ defined on $(U_i \cap
  U_j) \times [0,1]$. Thus on $U_i \cap U_j$, we have
  $d\ms{H}_i(\xi_i) = d\ms{H}_j(\xi_j)$, so $\eta$ is
  well-defined. Since $\eta$ is locally exact, it follows that it is
  closed. By \eqref{eq:def-norm}, we obtain
  \begin{displaymath}
    \norm{\xi - \eta} \leq \max_{1 \leq i \leq m} \norm{\xi_i -
      d\ms{H}_i(\xi_i)} \leq \max_{1 \leq i \leq m} \norm{d\xi_i} \leq \norm{d\xi}.
  \end{displaymath}
  This completes the proof of the proposition.
\end{proof}

\subsection{Change of Riemannian metric}
\label{sec:change}

Denote by $\Omega$ the smooth volume form preserved by the flow and
let $\ms{R}$ be the Riemannian metric which induces $\Omega$.

Our goal is to show that relative to some Riemannian metric the area
of the parallelogram $Tf_{-t}(v \wedge w)$ grows as $(\mu_+^{n_s-1}
\lambda_+^{n_u-1})^t$, where $v \in E^{ss}$, $w \in E^{uu}$, and $t
\geq 0$. To do this, it will be convenient to switch from the original
Riemannian metric $\ms{R}$ to a new metric $\ms{R}'$ with respect to
which $X$ is a unit vector and $E^c \oplus E^{ss} \oplus E^{uu}$ is an
orthogonal splitting. This metric can in general be only continuous
and the corresponding volume form $\Omega'$ may not be invariant with
respect to the flow. We will show that this does not present a
problem.

Let $\ms{R}'$ be as above and let $\Omega'$ be the Riemannian volume
form induced by $\ms{R}'$. Since $\Omega$ and $\Omega'$ are both
volume forms, there exists a positive continuous function $\phi$ such
that $\Omega' = \phi \: \Omega$. Let
\begin{displaymath}
  L = \frac{\max_M \phi}{\min_M \phi}.
\end{displaymath}
Denote the norms of tangent vectors (and their wedge products) with
respect to $\ms{R}$ and $\ms{R}'$ by $\norm{\cdot}$ and
$\norm{\cdot}'$, respectively. By compactness of $M$ there exist $b_-,
b_+ > 0$ such that
\begin{displaymath}
  b_- \norm{v} \leq \norm{v}' \leq b_+ \norm{v},
\end{displaymath}
for all $v \in TM$. Observe that for $v \in E^{ss}$, $w \in
E^{uu}$, and $t \geq 0$, we have
\begin{displaymath}
  \norm{Tf_t(v)}' \leq bc \mu_+^t \norm{v}' \quad \text{and} \quad 
  \norm{Tf_t(w)}' \leq bc \lambda_+^t \norm{w}',
\end{displaymath}
where $b = b_+/b_-$.
It is easy to check that
\begin{displaymath}
  f_t^\ast \Omega' = \frac{\phi \circ f_t}{\phi} \Omega',
\end{displaymath}
for all $t \in \R$. Thus for any $n$-dimensional parallelepiped $\Pi$
in a tangent space to $M$ and $t \in \R$, we have
\begin{displaymath}
  \norm{\Pi}' \leq L \norm{Tf_t(\Pi)}'
\end{displaymath}

\begin{lem}       \label{lem:backward}
  If $\Phi = \{ f_t \}$ is a volume preserving Anosov flow with
  constants defined in \eqref{eq:ss} and \eqref{eq:uu}, then
  \begin{displaymath}
    \norm{Tf_{-t}(v \wedge w)}' \leq L (bc)^{n-3} (\mu_+^{n_s-1}
    \lambda_+^{n_u-1})^t \norm{v \wedge w}',
  \end{displaymath}
  for all $v \in E^{ss}$, $w \in E^{uu}$ and $t \geq 0$.
\end{lem}

\begin{proof}
  Let $v \in E^{ss}$, $w \in E^{uu}$ and $t \geq 0$ be arbitrary. Set
  $v_1 = Tf_{-t}(v)$ and $w_1 = Tf_{-t}(w)$. Choose vectors $v_2,
  \ldots, v_{n_s} \in E^{ss}$ and $w_2, \ldots, w_{n_u} \in E^{uu}$
  such that, relative to $\ms{R}'$, $(X, v_1, \ldots, v_{n_s}, w_1,
  \ldots, w_{n_u})$ is an orthogonal basis of the corresponding
  tangent space and $v_i, w_i$ are all of unit length, for $i \geq
  2$. Then:
  \begin{align*}
    \norm{Tf_{-t}(v \wedge w)}' & = \norm{v_1 \wedge w_1}' \\
    & = \norm{X \wedge v_1 \wedge \cdots \wedge
      v_{n_s} \wedge w_1 \wedge \cdots \wedge w_{n_u}}' \\
    & = \norm{\Pi}' \\
    & \leq L \norm{Tf_t(\Pi)}' \\
    & = L \norm{Tf_t(X \wedge v_1 \wedge \cdots \wedge
      v_{n_s} \wedge w_1 \wedge \cdots \wedge w_{n_u})}' \\
    & \leq L \norm{Tf_t(v_1 \wedge w_1)}' \norm{Tf_t(X \wedge v_2 \wedge
      \cdots \wedge v_{n_s} \wedge w_2 \wedge \cdots \wedge w_{n_u})}' \\
    & \leq L \norm{v \wedge w}' (bc)^{n_s-1} \mu_+^{(n_s-1) t} \cdot
   (b c)^{n_u-1} \lambda_+^{(n_u-1)t} \\
    & = L (bc)^{n-3} \mu_+^{(n_s-1) t} \lambda_+^{(n_u-1)t} \norm{v \wedge w}'. \qedhere
  \end{align*}
\end{proof}

In the remainder of the paper we will always be working with $\ms{R}'$
as the underlying Riemannian metric on $M$; the norms of tangent
vectors and differential forms are taken relative to $\ms{R}'$ and
will be denoted by the symbol $\norm{\cdot}$ (thus slightly abusing
notation for the sake of keeping it less cumbersome).

\section{Proof of the main theorem}
\label{sec:proof}

We will construct a smooth closed 1-form $\eta$ such that $u =
\eta(X) > 0$. Assuming for a moment that such a form has been
found, the proof can be completed as follows. Define
\begin{displaymath}
  \tilde{X} = \frac{1}{u} X.
\end{displaymath}
Then $\tilde{X}$ is an Anosov vector field~\cite{anosov+sinai+67} and $\eta(\tilde{X}) =
1$. Thus the Lie derivative of $\eta$ with respect to $\tilde{X}$ satisfies
\begin{displaymath}
  L_{\tilde{X}} \eta = (d i_{\tilde{X}} + i_{\tilde{X}} d) \eta = 0,
\end{displaymath}
which implies that $\eta$ is invariant with respect to the flow. It
follows that its kernel $\text{Ker}(\eta)$ is an invariant codimension
one distribution transverse to the flow, so $\text{Ker}(\eta)$ is
forced to be the sum $\tilde{E}^{su} = \tilde{E}^{ss} \oplus
\tilde{E}^{uu}$ of the strong stable and strong unstable bundles of
$\tilde{X}$. Since $\eta$ is closed, $\tilde{E}^{su}$ is uniquely
integrable, so by Plante~\cite{plante72}, $\tilde{X}$ admits a global
cross section $\Sigma$, which is also a global cross section for $X$.

So it remains to construct a smooth 1-form $\eta$ with $\eta(X) > 0$,
which will be done in three steps. In the first two steps we construct
a smooth 1-form $\xi$ such that
\begin{displaymath}
  \xi(X) = 1 \qquad \text{and} \qquad \norm{d \xi} < 1.
\end{displaymath}
The third step consists of approximating $\xi$ by a smooth closed
1-form using Proposition~\ref{prop:forms}. \\

\noindent
\paragraph{\textbf{Step 1}} Let $\veps > 0$ be arbitrary. As in
\S\ref{sec:local-sections}, for $\veps < \veps_\ast$, we can cover $M$
be smooth flow boxes $V_1^\veps, \ldots, V_\ell^\veps$, with $\ell$
independent of $\veps$, such that with respect to the Hausdorff
distance each $V_i^\veps$ is close to a fixed open set $V_i$. In
addition, we have smooth functions $\tau_i^\veps : V_i^\veps \to \R$
such that
\begin{equation}   \label{eq:taus}
  X\tau_i^\veps = 1, \qquad \norm{d\tau_i^\veps \! \restriction_{E^{ss}}} \leq A \veps^\theta, 
    \quad \text{and} \quad \norm{d\tau_i^\veps} \leq B \veps^{\theta-1},
\end{equation}
where $A, B$ are constants independent of $\veps$.

Let $\{\psi_i^\veps \}$ be a smooth partition of unity subordinate to
the cover $\{ V_i^\veps \}$. Since $\ell$ and the sizes of the sets
$V_i^\veps$ are independent of $\veps$, there is a constant $C > 0$
also independent of $\veps$ such that
\begin{equation}    \label{eq:psi}
  \sum_{i=1}^\ell \norm{d \psi_i^\veps} \leq C,
\end{equation}
for all $0 < \veps < \veps_\ast$. Define
\begin{displaymath}
  \xi_0^\veps = \sum_{i=1}^\ell \psi_i^\veps \: d\tau_i^\veps.
\end{displaymath}
\begin{lem}       \label{lem:xi0}
  $\xi_0^\veps$ is a smooth 1-form on $M$ with the following
properties:

\begin{itemize}

\item[(a)] $\xi_0^\veps(X) = 1$;

\item[(b)] $\norm{\xi_0^\veps \! \restriction_{E^{ss}}} \leq A \veps^\theta$,
    for all $0 < \veps < \veps_\ast$;

\item[(c)] $\abs{d\xi_0^\veps (v,w)} \leq D \veps^\theta$, for all unit vectors $v, w
  \in E^{cs}$ and $0 < \veps < \veps_\ast$, where $D$ is a constant independent
  of $\veps$;

\item[(d)] $\norm{d \xi_0^\veps} \leq K \veps^{\theta-1}$, for every
  $0 < \veps < \veps_\ast$, where $K$ is a constant independent of $\veps$.

\end{itemize}
\end{lem}

\begin{proof}
  Part (a) is clear. Part (b) follows easily from the second
  inequality in \eqref{eq:taus}. To prove (c), first note that if
  $\beta$ is a bilinear form on an inner product space $E$ which
  splits into two orthogonal subspaces $E_1$ and $E_2$ and $\beta_{ij}
  = \beta \!  \restriction_{E_i \times E_j}$, then
  \begin{equation}    \label{eq:bilinear}
    \norm{\beta} \leq \norm{\beta_{11}} + 2 \norm{\beta_{12}} + \norm{\beta_{22}}.
  \end{equation}
  We fix $p \in M$ and take $\beta = d_p\xi_0^\veps$, $E_1 = E^c_p$ and $E_2 =
  E^{ss}_p$. Since
  \begin{displaymath}
    d\xi_0^\veps = \sum_{i=1}^\ell d\psi_i^\veps \wedge d\tau_i^\veps,
  \end{displaymath}
  if $v, w \in E^{ss}$ are unit vectors, then
  \begin{displaymath}
    \abs{d\xi_0^\veps(v,w)} \leq 2 C \veps^\theta,
  \end{displaymath}
  and
  \begin{displaymath}
    \abs{d\xi_0^\veps(X,v)} = \abs{\sum_i d\psi_i^\veps(X) \:
      d\tau_i^\veps(v) - d\psi_i^\veps(v) \: d\tau_i^\veps(X)} 
    = \abs{\sum_i d\psi_i^\veps(X) \:
      d\tau_i^\veps(v)} \leq A C \veps^\theta,
  \end{displaymath}
  where we used $d\tau_i^\veps(X) = 1$ and $\sum_i d\psi_i^\veps(v) =
  d \left(\sum_i \psi_i^\veps\right)(v) = 0$.  Thus by
  \eqref{eq:bilinear} we can take $D = (A+2) C \veps$ in (c). Finally,
  \eqref{eq:psi} and $\norm{d\tau_i^\veps} \leq B \veps^{\theta-1}$
  imply
  \begin{displaymath}
    \norm{d\xi_0^\veps} = \norm{\sum_{i=1}^\ell d\psi_i^\veps \wedge
      d\tau_i^\veps} \leq B C \veps^{\theta-1},
  \end{displaymath}
  so (d) holds with $K = B C$. 
\end{proof}

In summary, for every $0 < \veps < \veps_\ast$ we have a 1-form which is small on
$E^{ss}$, whose exterior derivative is small when restricted to
$E^{cs}$, but whose overall norm grows as $\veps^{\theta-1}$, as
$\veps \to 0$. \\

\noindent
\paragraph{\textbf{Step 2}} To remedy the problem represented by part
(d) of the previous lemma, we flow backwards and use
Lemma~\ref{lem:backward}. Let $t > 0$ be large (how large will be
specified shortly) and set
\begin{displaymath}
  \xi_t^\veps = f_{-t}^\ast \xi_0^\veps.
\end{displaymath}
\begin{lem}
  There exists a constant $H > 0$ independent of $\veps$ and $t$ such
  that
  \begin{displaymath}
    \norm{d\xi_t^\veps} \leq \max\left\{ H \veps^{\theta-1} (\mu_+^{n_s-1}
    \lambda_+^{n_u-1})^t, H \veps^\theta \mu_-^{-2t}  \right\}
  \end{displaymath}
\end{lem}

\begin{proof}
  Observe that if $\beta$ is a bilinear form as in the proof of the
  previous lemma and $\norm{\beta_{22}} \leq \norm{\beta_{ij}}$, for
  all $i, j = 1, 2$, then
  \begin{displaymath}
    \norm{\beta} \leq 4 \max(\norm{\beta_{11}}, \norm{\beta_{12}}).
  \end{displaymath}
  Fix $p \in M$ and take $\beta = d_p \xi_t^\veps$, $E = T_p M$, $E_1 =
  E^{cs}_p$ and $E_2 = E^{uu}_p$. Since $d \xi_t^\veps = f_{-t}^\ast
  (d \xi_0^\veps)$ and $\norm{Tf_{-t} \! \restriction_{E^{cs}}} \leq c
  \mu_-^{-t}$, for $t > 0$, Lemmas \ref{lem:xi0} and
  \ref{lem:backward} imply
  \begin{displaymath}
    \norm{\beta_{11}} \leq c^2 D \veps^\theta \mu_-^{-2t} \qquad \text{and} \qquad
    \norm{\beta_{12}} \leq K L \veps^{\theta-1} (bc)^{n-3} (\mu_+^{n_s-1}
    \lambda_+^{n_u-1})^t.
  \end{displaymath}
  Note that the second inequality holds because if $v \in E^{cs}$ and
  $w \in E^{uu}$, then $\abs{d \xi_t^\veps(v,w)} = \abs{d\xi_0
    (T_{-t}(v \wedge w))}$ is largest if $v \in E^{ss}$. Observe also
  that $\norm{\beta_{22}}$ is smaller than $\norm{\beta_{ij}}$ ($i, j
  = 1, 2$) since the flow contracts strong unstable vectors in
  negative time.

  The proof of the lemma is now complete with $H = 4 \max(c^2 D, K L
  (bc)^{n-3})$.
\end{proof}

Since $\xi_t^{\veps}(X) = 1$, for all $0 < \veps < \veps_\ast$ and $t > 0$, the
following lemma will complete the construction of the desired form
$\xi$.

\begin{lem}
  There exist $0 < \veps < \veps_\ast$ and $t > 0$ such that $\norm{d\xi_t^\veps} < 1$.
\end{lem}

\begin{proof}
  By the previous Lemma, it suffices to show that the following system
  of inequalities
  \begin{align*}
    H \veps^{\theta-1} (\mu_+^{n_s-1} \lambda_+^{n_u-1})^t & < 1 \\
    H \veps^\theta \mu_-^{-2t} & < 1
  \end{align*}
  admits a solution $\veps, t > 0$. Solving the first inequality for
  $t$ we obtain
  \begin{equation}       \label{eq:t-}
    t > \frac{(1-\theta) \log \veps - \log H}{(n_s-1) \log \mu_+ +
      (n_u-1) \log \lambda_+}.
  \end{equation}
  The second inequality is equivalent to
  \begin{equation}        \label{eq:t+}
    t < \frac{\log H + \theta \log \veps}{2 \log \mu_-}.
  \end{equation}
  There exists $t$ with these properties if and only if we can find
  $0 < \veps < \veps_\ast$ such that
  \begin{displaymath}
    \frac{(1-\theta) \log \veps - \log H}{(n_s-1) \log \mu_+ +
      (n_u-1) \log \lambda_+} < \frac{\log H + \theta \log \veps}{2 \log \mu_-}.
  \end{displaymath}
  Since $H$ is fixed and $\veps$ is small, this is possible
  if
  \begin{displaymath}
    \frac{(1-\theta) \log \veps}{(n_s-1) \log \mu_+ +
      (n_u-1) \log \lambda_+} < \frac{\theta \log \veps}{2 \log \mu_-},
  \end{displaymath}
  which is equivalent to \eqref{eq:mainthm}.
\end{proof}

\noindent

\paragraph{\textbf{Step 3}} 
Choose $0 < \veps < \veps_\ast$ and $t > 0$ such that $\norm{d\xi_t^\veps} < 1$ and
set
\begin{displaymath}
  \xi = \xi_t^\veps.
\end{displaymath}
By Proposition~\ref{prop:forms} there exists a smooth closed 1-form
$\eta$ such that
\begin{displaymath}
  \norm{\xi - \eta} \leq \norm{d\xi} < 1.
\end{displaymath}
Since $\xi(X) = 1$, it follows that $\eta(X) > 0$. This completes the
construction of $\eta$ and concludes the proof of the theorem. \qed


\bibliographystyle{amsalpha} 





\end{document}